\documentclass{article}

\usepackage{amsmath,amsfonts,amssymb,amstext,bbm,amsthm}

\usepackage{natbib}

\theoremstyle{definition}

\newtheorem{example}{Example}[section]

\theoremstyle{plain}

\newtheorem{corollary}{Corollary}[section]
\newtheorem{lemma}{Lemma}[section]

\newtheorem{proposition}{Proposition}[section]

\usepackage{float}
\floatstyle{ruled}
\newfloat{algorithm}{tbp}{loa}
\floatname{algorithm}{Algorithm}

\usepackage{multirow}

\usepackage{balance} 

\newcommand{\allitems}{\mathcal{M}}
\newcommand{\allagents}{\mathcal{N}}

\newcommand{\incomevector}{\mathbf{t}}

\newcommand{\range}[2]{\in\{#1,\dots,#2\}}

\newcommand{\partition}[2]{\textsc{Partition}(#1,#2)}
\newcommand{\union}[2]{\textsc{Union}(#1,#2)}
\newcommand{\maxmin}[3]{\left[\frac{#2}{#3}\right]#1}
\newcommand{\wmms}[3]{\text{WMMS}(#1,#2,#3)}
\newcommand{\bmms}[2]{\text{BMMS}(#1,#2)}

\begin{document}

\title{The Maximin Share Dominance Relation}

\author{Erel Segal-Halevi}

\maketitle

\begin{abstract}
Given a finite set $X$ and an ordering $\succeq$ over its subsets, 
the $l$-out-of-$d$ maximin-share of $X$
is the maximal (by $\succeq$) subset of $X$ 
that can be constructed by partitioning $X$ into $d$ parts and picking the worst union of $l$ parts.
A pair of integers $(l,d)$ dominates
a pair $(l',d')$ if, for any set $X$ and ordering $\succeq$, 
the $l$-out-of-$d$ maximin-share of $X$
is at least as good (by $\succeq$)
as the $l'$-out-of-$d'$ maximin-share of $X$.
This note presents a necessary and sufficient condition for deciding whether 
a given pair of integers dominates another pair,
and an algorithm for finding all non-dominated pairs.
It compares the $l$-out-of-$d$ maximin-share to some other criteria for fair allocation of indivisible objects among people with different entitlements.
\end{abstract}

\section{Introduction}
The concept of \emph{maximin share} 
was invented by \citet{budish2011combinatorial}
in his study of fair allocation of indivisible items.

Suppose first that $m$ identical items have to be allocated fairly among $n$ people. Ideally, each person should receive $m/n$ items, but this may be impossible if $m$ is not a multiple by $n$, since the items are indivisible. 
A natural second-best fairness criterion is to round $m/n$ down to the nearest integer, and give each person at least $\lfloor m/n \rfloor$ items. Receiving less than $\lfloor m/n \rfloor$ items is ``too unfair'' --- it is an unfairness that is not justified by the indivisibility of the items.

Suppose now that the items are different, and each item has a different value. Now, rounding down to the nearest integer may not be the right solution.
For example, suppose $n=3$ and $m=5$ and the items' values are $1, 3, 5, 6, 9$. The sum of values is $24$, and it is divisible by $3$, so ideally we would like to give each person a value of at least $8$, but this is not possible. The largest value that can be guaranteed to all three agents is $7$, by the partition $\{1,6\},\{3,5\},\{9\}$.
So here, $7$ is the total value divided by $3$ ``rounded down to the nearest item''.
The set $\{1,6\}$ attaining this maximin value 
is called the ``1-out-of-$3$ maximin-share (MMS)'' ---
it is the best subset of items that can be constructed by partitioning the original set into $3$ parts and taking the least valuable part
(see Section \ref{sec:model} for formal definitions).

By definition, when all agents agree on the values of the different items, it is always possible to give each agent at least his/her 1-out-of-$n$ MMS (this is not necessarily true when agents disagree; see Section \ref{sec:related} for details).

However, the 1-out-of-$n$ maximin-share makes sense only when all agents have the same rights.
In many cases, different agents may have different rights. 
For example, suppose Alice and Bob have a partnership where Alice has 40\% and Bob has 60\% of the shares. 
If the partnership is dissolved and they want to divide the property among them, Alice would naturally like to get at least her \emph{2-out-of-5 MMS} --- the best subset that can be constructed by partitioning the original set into $5$ parts and taking the $2$ least valuable ones. 
In the above example, it is easy to see that the 2-out-of-5 MMS is $\{1,3\}$.

But Alice can rightfully complain that giving her $\{1,3\}$ is not fair: her entitlement is 40\%, which is larger than $1/3$. Therefore she should get at least the 1-out-of-3 MMS, which --- as shown above --- is $\{1,6\}$.
This example illustrates an interesting property of the MMS:
it is possible that $l/d < l'/d'$, but still the $l$-out-of-$d$ MMS is better than the $l'$-out-of-$d'$ MMS!
Similarly, Bob should get at least his 3-out-of-5 MMS, which is $\{1,3,5\}$, but also at least his 1-out-of-2 MMS, which is $\{9,3\}$.
Again, $\{9,3\}$ is better although $1/2 < 3/5$.

In the above example, it is possible to satisfy both fairness requirements of both agents, for example by giving $\{1,6\}$ to Alice and $\{3,5,9\}$ to Bob.
But are these really the only ``relevant'' fairness conditions? 
There are infinitely many fractions that are smaller than $2/5$; 
how many of them should we verify, in order to ensure that the allocation is fair for Alice? 
There are infinitely many fractions that are smaller than $3/5$; 
how many of them should we verify in order to ensure that the allocation is fair for Bob?

To answer this question, we define a dominance relation, which is a partial order on the pairs of positive integers.
We say that a pair $(l,d)$ \emph{dominates}
a pair $(l',d')$, if the $l$-out-of-$d$ MMS 
is better than the $l'$-out-of-$d'$ MMS 
for any set of items with any values.
If $(l,d)$ dominates
$(l',d')$, 
and both $l/d$ and $l'/d'$ are smaller than $0.4$,
then Alice does not need to check the $l'$-out-of-$d'$ MMS in order to verify that her allocation is fair;
it is sufficient that she checks the $l$-out-of-$d$ MMS.

The dominance relation is only a partial order. The above example shows that $(2,5)$ does not dominate $(1,3)$;
it is easy to construct an example showing that $(1,3)$ does not dominate $(2,5)$ either (for example, with $5$ identical items, the 2-out-of-5 MMS is better than the 1-out-of-3 MMS).
The present note shows a necessary and sufficient condition for deciding whether a given pair dominates another given pair. The condition is constructive and implies an efficient algorithm for enumerating all MMS fairness conditions that are not dominated.

\section{Notation}
\label{sec:model}
Let $X$ be a finite set. 
Let $\succeq$ be a weak ordering on the subsets of $X$. We assume that $\succeq$ is (weakly) monotonically increasing, i.e.:
$X_1\supseteq X_2$ implies $X_1 \succeq X_2$.

As a special case, $X$ may be a set of non-negative numbers,
and $\succeq$ may order the subsets of $X$ 
according to their sum, so for example $\{9\} \succ \{1,5\} \succeq \{6\} \succ \{1,3\}$ etc. We call such an ordering \emph{additive}.
In general, we allow $X$ to be any set and $\succeq$ to be any monotonically-increasing ordering on its subsets --- not necessarily additive.

For every integer $d\geq 1$, $\partition{X}{d}$ denotes the set of all partitions of $X$ into $d$ subsets  (some possibly empty).

For every vector $\mathbf{Y}=(Y_1,\ldots,Y_d)$ of sets,
and every integer $l\range{1}{d}$, 
$\union{\mathbf{Y}}{l}$ denotes the set of all unions of exactly $l$ different sets from $\mathbf{Y}$, $Y_{j1}\cup Y_{j2} \cup \cdots \cup Y_{jl}$.

For every set $X$ and integers $l,d$ with $1\leq l\leq d$, the \emph{$l$-out-of-$d$ maximin-share} of $X$ is denoted $\maxmin{X}{l}{d}$ and defined as:
\begin{align*}
\maxmin{X}{l}{d}
~~
:= 
~~
\max_{\mathbf{Y}\in \partition{X}{d}}
~~
\min_{Z\in \union{\mathbf{Y}}{l}}
~~
Z
\end{align*}
where $\max,\min$ are based on the ordering $\succeq$.

If $\succeq$ is not a strict order, then the minimum and maximum may be multi-valued, in which case there may be several different sets that qualify as $\maxmin{X}{l}{d}$. These sets are all equivalent by $\succeq$, so there should be no ambiguity in expressions such as
$\maxmin{X}{l}{d} \succeq Z$ or $Z\succeq \maxmin{X}{l}{d}$.

Let $(l,d)$ and $(l',d')$ be two pairs of integers with $1\leq l\leq d$ and $1\leq l'\leq d'$.
We say that 
\emph{$(l,d)$ dominates $(l',d')$}
if, for any finite set $X$ with a subset-ordering $\succeq$:
\begin{align*}
\maxmin{X}{l}{d} \succeq \maxmin{X}{l'}{d'}.
\end{align*}
It is easy to see that, whenever $l\geq l'$ and $d\leq d'$, the pair $(l,d)$ dominates $(l',d')$. 
So for example, $(5,6)$ dominates $(5,7)$ which dominates $(4,7)$.
Deciding whether a domination holds becomes more challenging when $l$ and $d$ change in the same direction. 

\section{Characterization of dominance}
Recall that, by the integer division theorem, for every two positive integers $d', d$, there exists a unique pair of integers, $q\geq 1$ and $r\in\{0,\ldots,d-1\}$, such that $d' = q d - r$.

\begin{lemma}
\label{lem:mms-char}
Let $l,d,l',d'$ be integers such that $1\leq l\leq d$ and $1\leq l'\leq d'$. 

Let $q\geq 1$ and $r\in\{0,\ldots,d-1\}$ be the unique integers for which $d' = q d - r$.

Then $(l,d)$ dominates $(l',d')$ if and only if $q l - \min(l,r) \geq l'$.


\end{lemma}
The following running example is used to illustrate the proof. Let $l=2,~d=3,~d'=7$, so $d' = 3\cdot d - 2$ and $q=3,~r=2$.

Taking $l' = 3\cdot l - 2 = 4$ yields that 
$(2,3)$ dominates $(4,7)$.

In contrast, taking $l'=5$,
yields that 
$(2,3)$ does not dominate $(5,7)$.

\begin{proof}[Proof of Lemma \ref{lem:mms-char}]

~\\
\noindent
{\textbf{If part.}}

We assume that $q l - \min(l,r) \geq l'$.
We should prove that,
for every finite set $X$ with a subset-ordering, $\maxmin{X}{l}{d} \succeq \maxmin{X}{l'}{d'}$.

Let $W' := \maxmin{X}{l'}{d'}$.
By definition of the maximin share, 
there exists a partition $\mathbf{Y'}\in \partition{X}{d'}$,
such that $W' = \min_{\succ} \union{\mathbf{Y'}}{l'}$.
W.l.o.g. denote the bundles in that union by $Y'_1,\ldots,Y'_{l'}$
(in the running example, $W'$ is a union of $4$ bundles from a partition of $X$ into 7 bundles).

We now construct a new partition $\mathbf{Y}\in \partition{X}{d}$,
in which each part is a union of zero or more bundles from  $\mathbf{Y'}$.
In the new partition there will be $l$ ``small parts'' and $d-l$ ``large parts'':
the small parts will be unions of a small number of bundles from  $Y'_1,\ldots,Y'_{l'}$,
and the large parts will be unions of a larger number of bundles from  $Y'_{l'+1},\ldots,Y'_{d'}$.
The construction is slightly different depending on whether or not $l\geq r$.

If $l\geq r$, then $\mathbf{Y}$ is constructed as follows.

\begin{itemize}
\item
Each of the $l$ small parts is union of \emph{at most} $q$ bundles from $Y'_1,\ldots,Y'_{l'}$
(in the running example, the small parts may be $Y_1 := Y'_1\cup Y'_2$ and $Y_2 := Y'_3\cup Y'_4$).
This is possible since by assumption $ql \geq l' + \min(l,r) \geq l'$.
Note that some of these $l$ parts may be empty.
\item 
Each of the $d-l$ large parts is a union of \emph{at least} $q$ bundles from  $Y'_{l'+1},\ldots,Y'_{d'}$
(in the running example, the large part may be $Y_3 := Y'_5\cup Y'_6\cup Y'_7$).
This is possible since by assumption $q(d-l) = qd - ql = d' + r - (ql) \leq d' + r - (l'+\min(r,l)) = d' + r - (l'+r) = d'-l'$.
\end{itemize}
If $l< r$, then $\mathbf{Y}$ is constructed as follows.
\begin{itemize}
\item 
Each of the $l$ small parts is union of \emph{at most} $q-1$ bundles from $Y'_1,\ldots,Y'_{l'}$.
This is possible since by assumption $(q-1)l = ql - l = ql - \min(l,r) \geq l'$.
\item 
Each of the $d-l$ large parts is a union of \emph{at least} $q-1$ bundles from  $Y'_{l'+1},\ldots,Y'_{d'}$.
This is possible since by assumption $(q-1)(d-l) = (qd-d) - (ql-l) = (d'+r-d) - (ql-\min(r,l)) < (d') - (l')$.
\end{itemize}
The important property of this construction, in both cases, is that each small part contains at most as many bundles from $\mathbf{Y'}$ as any large part.

Let $W$ be any bundle in $\union{\mathbf{Y}}{l}$.
Then $W$ is a union of at least $l'$ bundles from $\mathbf{Y'}$.
Indeed, if $W$ is a union of the $l$ small parts,
then by definition it is the union of the $l'$ bundles $Y'_1,\ldots,Y'_{l'}$ from $\mathbf{Y'}$.
Otherwise, $W$ is a union that contains some large parts. 
Since each large part contains at least as many bundles from  $\mathbf{Y'}$
as each small part,
$W$ is a union of \emph{at least} $l'$ bundles from $\mathbf{Y'}$
(in the running example, any 
union of $2$ parts from $\mathbf{Y}$, 
is a union of $4$ or $5$ bundles from $\mathbf{Y'}$).

In any case, $W$ contains a union of $l'$ bundles from $\mathbf{Y'}$;
denote this union by $W''$.
Note that $W\supseteq W''$, so by the monotonicity of the subset ordering, $W \succeq W''$.
By definition of the maximin share, $W'$ is the minimum element in 
$\union{\mathbf{Y'}}{l'}$.
But $W''\in\union{\mathbf{Y'}}{l'}$, so $W'' \succeq W'$.
By transitivity, $W\succeq W'$.

This is true for any bundle $W\in \union{\mathbf{Y}}{l}$.
Hence it is also true for the  minimum bundle in $\union{\mathbf{Y}}{l}$, which is by definition $\maxmin{X}{l}{d}$.

~\\
\noindent
{\textbf{Only-if part.}}

We assume that $q l - \min(l,r) < l'$.
We should show 
a finite set $X$ and a subset-ordering for which $\maxmin{X}{l}{d} \prec \maxmin{X}{l'}{d'}$.

Let $X$ be a set with $d'$ elements.
Let $\succeq$ be an ordering 
that considers only the cardinality (larger sets are better).%
\footnote{
If a strict ordering is wanted, 
then one can let $\succeq$ be an additive ordering,
determined by sum of values, 
and assign to each element of $X$ a distinct random value that is very close to $1$.
}
Obviously, $\maxmin{X}{l'}{d'}$ contains $l'$ items and its value is $l'$.

The most balanced partition of $X$ into $d$ parts
contains $r$ parts with $q-1$ items and $d-r$ parts with $q$ items.

If $l\geq r$, then the $l$ low-value parts contain the $r$ parts with $q-1$ items 
and $l-r$ additional parts with $q$ items, so their total value is:
$(q-1)r+q(l-r) = qr-r+ql-qr = ql-r = ql-\min(l,r) < l'$.

If $l < r$, the $l$ low-value parts have $q-1$ items each,
so their total value is:
$(q-1)l = ql-l = ql-\min(l,r) < l'$.

$\maxmin{X}{l}{d}$ is this union of $l$ parts. In both cases it is strictly worse than $\maxmin{X}{l'}{d'}$.
\end{proof}

Some special cases of Lemma \ref{lem:mms-char} are presented below.
\begin{corollary}
\label{lem:mms-cor}
Let $l,d,l',d'$ be integers such that $1\leq l\leq d$ and $1\leq l'\leq d'$. 
Then in each of the following cases, $(l,d)$ dominates $(l',d')$:

(a) When $l > l'$ and $d = d'$ (here $q=1, r=0$).

(b) When $l = l'$ and $d < d'$ (here $q>1$ so $ql-r =ql'-r > l'$).

(c) When $l' = l-r$ and $d' = d-r$ (here $q=1$).

(d) When $l'/d'$ is a non-reduced fraction and $l/d$ is its reduced fraction (here $q$ is the common factor of $d'$ and $l'$, so that $d'=qd$ and $l'=ql$, and $r=0$).
\end{corollary}

\section{Finding all non-dominated pairs}
Suppose Alice is involved in a partnership, where her entitlement is some fraction $a\in (0,1)$. 
The partnership is dissolved and Alice receives some subset of the previously shared items.
She would like to decide whether she received a fair share according to her entitlement.
At first glance, to verify this she should check the $l$-out-of-$d$ MMS for \emph{all} pairs $(l,d)$ such that $l/d \leq a$. However, there are infinitely many such pairs.
In fact, Alice should check only finitely many such pairs.
This is proved by the following lemma. 
\begin{lemma}
\label{lem:mms-bundlesize}
For every integers $h\geq 0$ and $1\leq l\leq d$, If $|X| \leq  d$ then
$
\maxmin{X}{l}{d}
\succeq
\maxmin{X}{l+h}{d+h}.
$
\end{lemma}
\begin{proof}
Let $W' := \maxmin{X}{l+h}{d+h}$.
By definition of the maximin share, 
there exists a partition $\mathbf{Y'}\in \partition{X}{d+h}$,
such that $W' = \min_\succ \union{\mathbf{Y'}}{l+h}$.

Since $|X|\leq d$, at most $d$ parts in $\mathbf{Y'}$ are non-empty.
Consider a new partition $\mathbf{Y}\in\partition{X}{d}$,
which contains all non-empty parts in $\mathbf{Y'}$.
Then, each bundle $W\in \union{\mathbf{Y}}{l}$
is also contained in $\union{\mathbf{Y'}}{l+h}$, so
$W\succeq W'$. 
This is true, in particular, when $W$ is the minimum bundle in $\union{\mathbf{Y}}{l}$, which is by definition $\maxmin{X}{l}{d}$.
\qed
\end{proof}
Lemma \ref{lem:mms-bundlesize} implies that, if the total number of allocated items is $m$, then Alice should only check pairs $(l,d)$ such that $d\leq m$ (in addition to $l/d\leq a$).
This is already a finite number of conditions to check.
Adding Lemma \ref{lem:mms-char} allows to reduce the number of conditions even further, using the following algorithm:
\begin{enumerate}
\item Let $D := \{1,\ldots,|X|\}$.
\item For each $d\in D$:
\begin{itemize}
\item Let $l_d := $ the largest integer for which $l_d / d \leq a$.
\end{itemize}
\item For each pair $d,d'\in D$ with $d\neq d'$:
\begin{itemize}
\item Using integer division, find $q\geq 1,r\range{0}{d-1}$ such that $d' = qd-r$.
\item If $q l_{d} - \min(l_d,r)  \geq l_{d'}$, then remove $d'$ from $D$.
\end{itemize}
\item For each remaining $d\in D$, check the subset $\maxmin{X}{l_d}{d}$.
\end{enumerate}
The filtering in step 2 is justified by Corollary \ref{lem:mms-cor}(a).
The filtering in step 3 is justified by the if part of Lemma \ref{lem:mms-char}. If the condition is satisfied, 
$(l_d,d)$ dominates $(l_{d'},d')$, so 
there is no need to check the subset $\maxmin{X}{l_{d'}}{d'}$.

By the only-if part of Lemma \ref{lem:mms-char}, all remaining conditions are independent (not implied by others).

\begin{example}
Let $a = 0.74$ and suppose $X$ contains $7$ items.
In step 2 of the above algorithm, the following seven pairs are generated:
$0/1$, $1/2$, $2/3$, $2/4$, $3/5$, $4/6$, $5/7$.
In step 3, the following pairs are filtered out:
\begin{itemize}
\item $0/1$ is filtered out by $2/3$ (with $q=1,r=2$).
\item $1/2$ is filtered out by $2/3$ (with $q=1,r=1$).
\item $2/4$ is filtered out by $2/3$ (with $q=2,r=2$).
\item $3/5$ is filtered out by $5/7$ (with $q=1,r=2$).
\item $4/6$ is filtered out by $2/3$ (with $q=2,r=0$).
\end{itemize}

In step 4, the following MMS bundles remain:
\begin{align*}
\maxmin{X}{2}{3}
&&
\maxmin{X}{5}{7}
\end{align*}
So, to verify that she got a fair share (by the MMS criterion), Alice should compare her bundle to these two subsets only.

As implied by Lemma \ref{lem:mms-char}, with $l=2,~l'=5,~d=3,~d'=7$, these two comparisons are independent.
For example, suppose Alice's preferences are additive. Then, if she values the items in $X$ at $\{1,1,1,0,0,0,0\}$ then $\maxmin{X}{2}{3}$ is worth 2 and $\maxmin{B}{5}{7}$ is worth 1 so the comparison to $\maxmin{X}{2}{3}$ is stronger, 
while if she values the items in $X$ at $\{1,1,1,1,1,1,1\}$ then
$\maxmin{X}{2}{3}$ is worth 4 and $\maxmin{X}{5}{7}$ is worth 5 so the comparison to $\maxmin{X}{5}{7}$ is stronger.
\end{example}

\section{Ordinal vs. Cardinal MMS}
Suppose we are given a set $\allitems$ of items, an order relation $\succeq$ on $2^{\allitems}$, and a vector of $n$ entitlements $\incomevector =(t_1,\ldots,t_n)$, with $\sum_{i=1}^n t_i = 1$.
Our goal is to find a partition $X_1\sqcup \cdots \sqcup X_n = \allitems$
that is ``fair'' with respect to the vector $\incomevector$.

Define a partition $\mathbf{X}$ as \emph{Ordinal-MMS (OMMS) fair for $i$} if, for every positive integers $l,d$ with $l/d\leq t_i$, $X_i \succeq \maxmin{\allitems}{l}{d}$. In the previous section we saw that, despite the infinite number of pairs $l,d$ satisfying with $l/d\leq t_i$, it is still possible to check in finite time whether a given allocation is OMMS-fair. Note that OMMS-fairness is based only on the order relation $\succeq$ on subsets of  $\allitems$ (hence the name ordinal).

In this section we consider a special case in which $\allitems$ is a multi-set of numbers. For any subset $Z\subseteq \allitems$, denote by $V(Z)$ the sum of elements of $Z$. The relation $\succeq$ orders subsets by their sum, so that $Z_1\succeq Z_2 \iff V(Z_1)\geq V(Z_2)$. 
Arguably, an ideal partition is a partition $\mathbf{X}$ into $n$ bundles in which, for each bundle $i$, $V(X_i) = t_i\cdot V(\allitems)$. In case such a partition does not exist, we would like a partition that is ``as close as possible'' to that ideal partition
. 
This is formalized by the following definition (it is based on the weighted-MMS notion of \citet{farhadi2019fair}):
\begin{align*}
\wmms{\allitems}{\mathbf{t}}{i}
~~
:= 
t_i\cdot
\left(
~~
\max_{\mathbf{Y}\in \partition{\allitems}{n}}
~~
\min_{j\in\allagents}
~~
\frac{V(Y_j)}{t_j}
\right)
\end{align*}
A partition $\mathbf{X}$ is called \emph{WMMS-fair for $i$} if $V(X_i)\geq \wmms{\allitems}{\mathbf{t}}{i}$.
Note that the WMMS crucially depends on the choice of the function $V$: different functions $V$ lead to different WMMS values, even if they induce the same ordering on the subsets of $Z$.

If all entitlements are equal ($t_i=1/n$ for all $i$), then $\wmms{\allitems}{\mathbf{t}}{i}$ is exactly the value of $\maxmin{\allitems}{1}{n}$, so WMMS-fairness and OMMS-fairness coincide. However, when the entitlements are different, these two notions are different --- none of them dominates the other one.

The following proposition shows that
WMMS-fairness may be strictly stronger than OMMS-fairness. 
\begin{proposition}
There are sets $\allitems$ and vectors $\mathbf{t}$ of size $n=2$,
for which, 
for all $i \in [n]$ and all positive integers $l,d$ with $l/d\leq t_i$:
\begin{align*}
\wmms{\allitems}{\mathbf{t}}{i}
>
V\left(\maxmin{\allitems}{l}{d}\right).
\end{align*}
\end{proposition}
\begin{proof}
Let $n=2$ and $\mathbf{t}=(0.4,0.6)$ and $\allitems = \{40,60\}$.
Then, 
$\wmms{\allitems}{\mathbf{t}}{1}=40$ and 
$\wmms{\allitems}{\mathbf{t}}{2}=60$
by the partition $\mathbf{Y}= (\{40\},\{60\})$, and this is the unique WMMS-fair partition.

Since $\allitems$ contains two items, Lemmas \ref{lem:mms-char} and \ref{lem:mms-bundlesize} imply that only bundles with $d\in\{1,2\}$ should be checked for the right-hand side.
Since $t_1 < 1/2$, the most valuable MMS bundle for $l/d \leq t_1$ is $\emptyset$.
Since $t_2 < 1$, the most valuable MMS bundle for $l/d \leq t_2$ is $\maxmin{\allitems}{1}{2} = \{40\}$. So:
\begin{align*}
\wmms{\allitems}{\mathbf{t}}{1} = 
40 > 0 = V(\emptyset)
\\
\wmms{\allitems}{\mathbf{t}}{2} - 60 > 40 = V(\maxmin{\allitems}{1}{2})
\end{align*}
Indeed, every partition in which $|X_2|\geq 1$ is OMMS-fair.
\end{proof}

The following proposition shows that WMMS-fairness may be strictly \emph{weaker} than OMMS-fairness.

\begin{proposition}
There are sets $\allitems$ and vectors $\mathbf{t}$ of size $n=3$,
for which, 
for some $i \in [n]$ and positive integers $l,d$ with $l/d\leq t_i$:
\begin{align*}
V\left(\maxmin{\allitems}{l}{d}\right)
>
\wmms{\allitems}{\mathbf{t}}{i}.
\end{align*}
\end{proposition}
\begin{proof}
Let $n=3$ and $\mathbf{t}=(0.6,0.2,0.2)$ and $\allitems = \{40,60\}$.
Then, in any partition of $\allitems$ into three parts, at least one part is empty, so $\wmms{\allitems}{\mathbf{t}}{i}
=0$ for all $i\in[3]$, and thus every allocation is WMMS-fair. 

However, with $i=1,l=1,d=2$, we have $\maxmin{\allitems}{1}{2}=\{40\}$, so:
\begin{align*}
V\left(\maxmin{\allitems}{1}{2}\right) = 40 > 0 = \wmms{\allitems}{\mathbf{t}}{1}
\end{align*}
Indeed, not every partition is OMMS-fair --- OMMS-fairness requires that $X_1$ contains at least one item.
\end{proof}

The problem with WMMS-fairness, as illustrated in the previous proposition, is that the WMMS of $i$ depends on the entire vector $\mathbf{t}$, rather than just on $t_i$. To fix this, I suggest the \emph{bipartite weighted maximin-share (BMMS)}, defined as:
\begin{align*}
\bmms{\allitems}{t_i}
~~
:= 
t_i\cdot
\left(
~~
\max_{(X,Y)\in \partition{\allitems}{2}}
~~
\min \left(\frac{V(X)}{t_i}, \frac{V(Y)}{1-t_i}\right)
~~
\right)
\end{align*}
In BMMS, in contrast to WMMS, $\allitems$ is always partitioned into two parts: a part for $i$ and a part for ``all the others''. The ratio between the values of these parts should be as near as possible to $t_i/(1-t_i)$.

When $n=2$, BMMS-fairness is equivalent to WMMS-fairness. However, they differ when $n\geq 3$. 

I believe that BMMS-fairness implies both  WMMS-fairness and OMMS-fairness, but have not yet had time to verify this.

Note that a WMMS-allocation, by definition, always exists. However, it is open whether an OMMS-allocation always exists, and whether a BMMS-allocation always exists. It is apparently open even for $n=2$.

\section{Related Work}
\label{sec:related}
\label{sec:future}
This note focused on a set with a single subset-ordering $\succeq$. 
It corresponds to a fair allocation problem in which all participants have \emph{equal preferences}, but may have \emph{different entitlements}.

The maximin share was introduced in the context of fair allocation among agents with \emph{equal entitlements} but \emph{different preferences}.
In this problem, there are $n$ participants with $n$ different subset-orderings, $\succeq_1,\ldots,\succeq_n$,
and each participant should receive a subset of $\allitems$ that satisfies the fairness condition with respect to  $\succeq_i$.
In this context,
\citet{Moulin1990Uniform}
 introduced the fairness criterion called \emph{positive preference externalities}, which means that an agent should weakly prefer his share to the share he would get if all agents had the same preferences.
This naturally lead to the definition of the $1$-out-of-$n$ MMS
\citep{budish2011combinatorial}. 
The generalization to $l$-out-of-$d$ MMS was done by
\citet{babaioff2019competitive,babaioff2019fair}.

With \emph{equal entitlements} and \emph{equal preferences}, WMMS-fairness and OMMS-fairness and BMMS-fairness are all equivalent, and they are all satisfied by the allocation defining the 1-out-of-$n$ maximin share.

With \emph{equal entitlements} but \emph{different preferences}, a 1-out-of-$n$ MMS allocation always exists when $n=2$ but may not exist when $n\geq 3$ \citep{procaccia2014fair}.
This discovery has lead to the definition of \emph{$r$-fraction MMS allocation}. Suppose each preference-ordering $\succeq_i$ is represented by a cardinal value-function $V_i$; then,  in an {$r$-fraction MMS allocation}, for all $i$, $V_i(X_i) \geq r\cdot V_i(\maxmin{X}{1}{n})$.
There are algorithms for finding an $r$-fraction MMS allocation for various fractions $r\in(0,1)$ and various assumptions on the preference orderings  \citep{amanatidis2016truthful,amanatidis2017approximation,barman2017approximation,ghodsi2018fair,kurokawa2018fair,barman2018groupwise,garg2018approximating,huang2019algorithmic}.

The case of \emph{different preferences} and \emph{different entitlements} was studied by \citet{farhadi2019fair,aziz2019maxmin}

\section{Acknowledgments}
Lemma \ref{lem:mms-char} was derived in order to answer a question by an anonymous referee of the AAMAS journal, to whom I am very grateful.

\bibliographystyle{apalike}
\bibliography{../erelsegal-halevi}

\end{document}